\documentclass[12pt,a4paper]{scrartcl}
\pdfoutput=1
\usepackage[utf8]{inputenc}
\usepackage[T1]{fontenc}
\usepackage{lmodern}
\usepackage[english]{babel}
\usepackage{amsmath}
\usepackage{amssymb}
\usepackage{amsthm}
\usepackage{mathtools}
\usepackage{setspace}
\usepackage{mleftright}
\usepackage[left=2.5cm,right=2.5cm,top=2.5cm,bottom=2.5cm]{geometry}
\DeclareMathAlphabet{\mathbbold}{U}{bbold}{m}{n}
\usepackage{hyperref}
\hypersetup{breaklinks=true, hidelinks}
\numberwithin{equation}{section}
\theoremstyle{definition}
\newtheorem{defi}{Definition}[section]

\newtheorem{trm}[defi]{Theorem}
\newtheorem{lem}[defi]{Lemma}
\newtheorem{cor}[defi]{Corollary}
\newtheorem{prop}[defi]{Proposition}
\newcommand{\p}[1]{\mleft(#1\mright)}
\newcommand{\cp}[1]{\mleft\{#1\mright\}}
\newcommand{\comp}[2]{{\normalfont(cf. \cite[#1]{#2})}}
\newcommand{\comptwo}[4]{{\normalfont(cf. \cite[#1]{#2}, \cite[#3]{#4})}}
\newcommand{\inN}{\in\mathbb{N}}
\newcommand{\seq}[2]{\p{#1_{#2}}_{#2\inN}}
\newcommand{\conv}[2]{#1\xrightarrow{\ \, \ }#2}
\newcommand{\GH}[2]{#1\xrightarrow{\mathcal{GH}}#2}
\newcommand{\uni}[2]{#1\xrightarrow{uni.}#2}
\newcommand{\Haus}[2]{#1\xrightarrow{\ \mathcal{H}\ }#2}
\begin{document}
\begin{center}
\Large{Convergence of Riemannian 2-manifolds under a uniform curvature and contractibility bound}\\[0.5cm]
\scriptsize{TOBIAS DOTT}
\end{center}
\begin{abstract}
We consider uniformly semi-locally 1-connected sequences of closed connected Riemannian 2-manifolds. In particular, we assume that the manifolds are homeomorphic to each other and that their total absolute curvature is uniformly bounded. The purpose of this paper is a description of the Gromov-Hausdorff limits of such sequences. Our work extends earlier investigations by Burago and Shioya.    
\end{abstract}
\let\thefootnote\relax
\let\thefootnote\relax\footnotetext{The author was partially supported by the DFG grant SPP 2026 (LY 95/3-2).}
\section{Introduction}
In the middle of the 20th century, Alexandrov developed the theory of \emph{surfaces with bounded curvature}. One of the central results states that every surface with bounded curvature can be obtained as the uniform limit of Riemannian 2-manifolds with uniformly bounded total absolute curvature \comp{p. 77}{Res93}.\\
From this result a natural question arises: Given a closed surface $S$ and $C>0$. What can we say about the Gromov-Hausdorff limits of Riemannian 2-manifolds $R_n$ satisfying the following properties:
\begin{itemize}
\item[1)] The space $R_n$ is homeomorphic to $S$.
\item[2)] The total absolute curvature of $R_n$ is at most $C$. 
\end{itemize}
\noindent
Burago was the first to address this question: From results in \cite[pp. 143-145]{Bur65} a geometric description for the case that $S$ is homeomorphic to $\mathbb{S}^2$ can be deduced. However, there are no proofs given in the paper and we note that our investigation is completely independent from this work. In the 1990s, Shioya gave a purely topological description of the limit spaces \comp{p. 1767}{Shi99}. Recently, the current author answered the curvature-free formulation of the question \comp{p. 2}{Dot24}.\\
The aim of the present paper is a description of the limit spaces under the following additional assumption:
\begin{itemize}
\item[3)] There is some uniform constant $\varepsilon>0$ such that the following statement applies: Every loop in $R_n$ of diameter at most $\varepsilon$ is contractible in $R_n$. \end{itemize}
We write $\mathcal{R}\p{S,C}$ for the class of all Riemannian 2-manifolds satisfying the first two conditions. If a sequence of metric spaces satisfies the third condition, then we call it \emph{uniformly semi-locally 1-connected}.\\
The third assumption is only used in order to simplify the presentation of the results. The procedure developed in \cite{Dot24} suggests a way to describe the spaces in the entire closure of $\mathcal{R}\p{S,C}$ on basis of our main result.\\
\noindent
In particular, we will address the geometric properties of the limit spaces.\\
Our description uses the language of Whyburn's  so called \emph{cyclic element theory} \comp{pp. 64-87}{Why42}: By a \emph{Peano space} $X$ we mean a compact, connected and locally connected metric space. A subset $A\subset X$ is called \emph{cyclicly connected} if every pair of points in $A$ lies on some topological 1-sphere in $A$. Provided $A$ is maximal with this property and contains more than one point, we denote $A$ as  \emph{maximal cyclic}.\\
The \emph{index} of a point $p\in X$ is defined as the number of connected components in $X\setminus\cp{p}$ and we denote it by $ind_X(p)$. A point with $ind_X(p)>1$ is called a \emph{cut point} of $X$ and we write $Cut_X$ for the set of all these points. If $p$ admits arbitrarily small open neighborhoods whose boundaries consist of exactly one point, then we say that $p$ is an \emph{endpoint} of $X$. We denote the set of all cut and endpoints of $X$ by $Sing_X$.\\
Combining Shioya's description in \cite[p. 1767]{Shi99} with that of the author in \cite[p. 12]{ Dot24}, we directly obtain the following result about the topological structure of the limit spaces:
\begin{trm}\label{trm_topo_struc}
Let $X$ be a space that can be obtained as the Gromov-Hausdorff limit of a uniformly semi-locally 1-connected sequence in $\mathcal{R}\p{S,C}$. Then the following statements apply:
\begin{itemize}
\item[1)] If $S$ is homeomorphic to $\mathbb{S}^2$, then all maximal cyclic subsets of $X$ are homeomorphic to $\mathbb{S}^2$.
\item[2)] If $S$ is not homeomorphic to $\mathbb{S}^2$, then one maximal cyclic subset of $X$ is homeomorphic to $S$ and all others are homeomorphic to $\mathbb{S}^2$.
\item[3)] Every maximal cyclic subset of $X$ contains only finitely many cut points of $X$. Only finitely many of them contain more than two cut points.
\item[4)] The following inequality holds:
\begin{align*}
\sum_{p\in X}\max\cp{ind_X(p)-2,0}\le \frac{C}{2\pi}.    
\end{align*} 
\end{itemize}
\end{trm}
\noindent
For the definition of uniform convergence of metric spaces we refer to Section \ref{sec_uni_conv}. The following curvature-free result will create the needed connection between Gromov-Hausdorff and uniform convergence. In particular, it marks the starting point of our investigation. 
\begin{trm}\label{trm_loc_uni_appr}
Let $X$ be a space that can be obtained as the Gromov-Hausdorff limit of a uniformly semi-locally 1-connected sequence of length spaces that are homeomorphic to $S$. Then for every $p\in X$ we either have $p\in \overline{Cut_X}$ or the sequence $X_n$ converges uniformly to $X$ on some open neighborhood of $p$. 
\end{trm}
\noindent
For our main result we need a deeper insight into the theory of surfaces with bounded curvature: The \emph{total curvature measure} on a Riemannian 2-manifold $R$ is defined by
\begin{align*}
\omega_R(A)\coloneqq\int_{A}K \,d\mathcal{H}^2 \end{align*} 
\noindent
for all Borel sets $A\subset X$ where $K$ denotes the Gaussian curvature and $\mathcal{H}^2$ the 2-dimensional Hausdorff measure on $R$. We call the total variation measure of this signed measure the \emph{total absolute curvature measure} on $R$ and denote it by $\lvert\omega_R\rvert$.\\
Let $X$ be a surface with bounded curvature. Then $X$ can be obtained as the uniform limit of Riemannian 2-manifolds $R_n$ such that the measures $\lvert\omega_{R_n}\rvert$ are uniformly bounded. The Banach-Alaoglu theorem implies that, after passing to a subsequence, the total curvature measures converge weakly to some signed measure on $X$. Indeed, the limit measure does not depend on the choice of the spaces $R_n$ \comp{p. 240-241}{AZ67} and we refer to it as the \emph{total curvature measure} $\omega_X$ on $X$.\\
The \emph{total angle} of a point $p\in X$ is defined by 
\begin{align*}
\theta_X(p)\coloneqq 2\pi-\omega_X(p).  \end{align*}
Now we are able to state the main result of this paper:
\begin{trm}\label{trm_main}
Let $X$ be a compact length space. If $X$ can be obtained as the Gromov-Hausdorff limit of a uniformly semi-locally 1-connected sequence in $\mathcal{R}\p{S,C}$, then the following statements apply: 
\begin{itemize}
\item[1)] The topological structure of $X$ is given as in Theorem \ref{trm_topo_struc}. 
\item[2)] Every maximal cyclic subset of $X$ is a closed surface with bounded curvature.
\item[3)] Let $\seq{T}{n}$ be an enumeration of the maximal cyclic subsets of $X$. Then the following inequality holds:
\begin{align*}
\sum_{n=1}^{\infty} \lvert\omega_{T_n}\rvert\p{T_n\setminus Cut_X}+\sum_{p\in Sing_X}2\pi\,\lvert ind_X(p)-2\rvert+\sum_{p\in Cut_X}\sum_{n=1}^{\infty}  \mathbbold{1}_{T_n}(p)\,\theta_{T_n}(p)\le C.    
\end{align*}
\end{itemize}
\end{trm}
\noindent
We denote the class of all compact length spaces meeting the aforementioned description by $\mathcal{L}\p{S,C}$.\\
Also the following converse statement holds:
\begin{trm}\label{trm_converse}
Let $X$ be a space in $\mathcal{L}\p{S,C}$. Then $X$ can be obtained as the Gromov-Hausdorff limit of a uniformly semi-locally 1-connected sequence $X_n\in\mathcal{R}\p{S,C_n}$ where $\lim\limits_{n\to\infty}C_n\le C$.      
\end{trm}
\noindent
\subsection{Organization}
In Section \ref{sec_pre} we give an introduction to the convergence of metric spaces and surfaces with bounded curvature. Section \ref{sec_loc} is devoted to the proof of Theorem \ref{trm_loc_uni_appr}. The main theorem is discussed in Section \ref{sec_main} and the converse statement is shown in the last section.
\subsection{Acknowledgments}
The author is grateful to Alexander Lytchak for enlightening discussions on the problem.
\section{Preliminaries}\label{sec_pre}
\subsection{Convergence of metric spaces}
In this subsection we discuss several types of convergence for metric spaces.
\subsubsection{Uniform convergence}\label{sec_uni_conv}
As we have already seen, uniform convergence plays an important role in connection with surfaces with bounded curvature. We  now give the formal definition of this type of convergence: Let $\seq{X}{n}$ and $X$ be metric spaces. Then we say that $X_n$ \emph{converges uniformly} to $X$ if there are homeomorphisms $\varphi_n\colon X\to X_n$ such that 
\begin{align*}
\conv{dis\p{\varphi_n}\coloneqq\sup_{x,y\in X} \lvert d_X\p{x,y}-d_{X_n}\p{\varphi_n(x),\varphi_n(y)}\rvert}{0}.   
\end{align*}
We write $\uni{X_n}{X}$ for uniform convergence. Provided the topological embeddings $\varphi_n$ are only defined on some subset $U\subset X$, then we say that $X_n$ \emph{converges uniformly} to $X$ \emph{on} $U$.\\
By an $\varepsilon$-equivalence between two metric spaces $X$ and $Y$, we mean a continuous map $f\colon X\to Y$ satisfying the following properties:
\begin{itemize}
\item[1)] There is a continuous map $g\colon Y\to X$ such that $g\circ f$ and $f\circ g$ are homotopic to $id_X$ and $id_Y$ respectively.
\item[2)] There is a choice of the corresponding homotopies $H_X$ and $H_Y$ such that 
\begin{align*}
d_Y\p{f(x),f\circ H_X\p{x,t}}<\varepsilon \quad\text{and}\quad d_Y\p{y,H_Y\p{y,t}}<\varepsilon   
\end{align*}
for all $x\in X$, $y\in Y$ and $t\in\left[0,1\right]$.
\end{itemize}
As a direct consequence of Jakobsche's 2-dimensional $\alpha$-approximation result in \cite[p. 2]{Jac83}, we derive the following uniform convergence criterion for closed surfaces:
\begin{trm}\label{trm_alpha}
Let $S$ be a closed surface. Moreover let $X$ and $\seq{X}{n}$ be metric spaces that are homeomorphic to $S$. If there are $\varepsilon_n$-equivalences $f_n\colon X\to X_n$, where  $\conv{\varepsilon_n}{0}$, then $\uni{X_n}{X}$. Moreover the homeomorphisms $\varphi_n\colon X\to X_n$ corresponding to the uniform convergence can be chosen such that $\conv{d\p{\varphi_n,f_n}}{0}$.    
\end{trm}
\subsubsection{Gromov-Hausdorff convergence}
The definition of the Gromov-Hausdorff distance and the Hausdorff  distance can be found in \cite[pp. 252, 254]{BBI01}. In the following we write $\GH{X_n}{X}$ for Gromov-Hausdorff convergence and $\Haus{A_n}{A}$ for Hausdorff convergence.\\Gromov-Hausdorff convergence of compact metric spaces can always be considered as Hausdorff convergence in some compact metric space:
\begin{prop}\label{prop_GH_H}\comp{pp. 64-65}{Gro81} Let $X_n$ be compact metric spaces and $\GH{X_n}{X}$. Then there is a compact metric space $Y$ and isometric embeddings $f_n\colon X_n\to Y$ and $f\colon X\to Y$ such that $\Haus{f_n\p{X_n}}{f(X)}$ in $Y$. 
\end{prop}
\noindent
This result enables us to talk about Hausdorff convergence of subsets $A_n\subset X_n$ to some subset $A\subset X$.\\
Now we consider Gromov-Hausdorff convergence under a uniform bound on the local connectivity: A sequence of metric spaces $\seq{X}{n}$ is called \emph{uniformly locally connected} if for every $\varepsilon>0$ there is $\delta>0$ and $N\inN$ such that the following statement applies for all $n\ge N$: Every pair of points  $x,y\in X_n$ with $d_{X_n}\p{x,y}<\delta$ lies in some compact connected subset of $X_n$ with diameter less than $\varepsilon$.\\
With regard to this property we present two results that go back to Whyburn. The first one deals with the convergence of decompositions:
\begin{lem}\label{lem_lim_int}\comptwo{p. 253}{BBI01}{p. 412-413}{Why35}
Let $\seq{X}{n}$ be a uniformly locally connected sequence of compact metric spaces and $\GH{X_n}{X}$. Moreover let $A_n,B_n\subset X$ be closed subsets with $A_n\cup B_n=X$. Then, after passing to a subsequence, there are $A,B\subset X$ such that $\Haus{A_n}{A}$, $\Haus{B_n}{B}$ and $\Haus{A_n\cap B_n}{A\cap B}$. If $\p{A_n\cap B_n}_{n\inN}$ is uniformly locally connected, then also the sequences $\seq{A}{n}$ and $\seq{B}{n}$.  \end{lem}
\noindent 
The second result describes the limits of closed topological discs:
\begin{prop}\label{prop_lim_disc}\comp{p. 421-422}{Why35}
Let $\seq{X}{n}$ be a uniformly locally connected sequence of metric spaces that are homeomorphic to the closed disc and $\GH{X_n}{X}$. Moreover we assume that $\p{\partial X_n}_{n\inN}$ is uniformly locally connected and $\Haus{\partial X_n}{J}$ where $diam(J)>0$. Then one maximal cyclic subset $T\subset X$ is homeomorphic to the closed disc and all others are homeomorphic to $\mathbb{S}^2$. Moreover we have $\partial T=J$.  
\end{prop} 
\noindent
\subsubsection{Measured Gromov-Hausdorff convergence}
Our notion of curvature is expressed by a measure. In this context the following notion of convergence for metric measure spaces will be useful: Let $\p{X_n,\mu_n}$ and $\p{X,\mu}$ be compact metric spaces together with finite Borel measures over them. Then we say that  $\p{X_n,\mu_n}$ Gromov-Hausdorff converges to $\p{X,\mu}$ if $\GH{X_n}{X}$ and there is a choice of the common ambient space $Y$ from Proposition \ref{prop_GH_H} such that the measures $\mu_n$ converge weakly to $\mu$ as measures over $Y$.\\
As a direct consequence of the Banach-Alaoglu theorem, we obtain the following result:
\begin{prop}
Let $\p{X_n,\mu_n}$ be compact metric spaces together with finite Borel measures over them and $\GH{X_n}{X}$. If the measures are uniformly bounded, then there is a finite Borel measure $\mu$ over $X$ such that, after passing to a subsequence, we may assume that $\GH{\p{X_n,\mu_n}}{\p{X,\mu}}$. 
\end{prop}
\subsection{Surfaces with bounded curvature}
The basics about surfaces with bounded curvature are discussed in the monograph \cite{AZ67}. In particular, we point out that the concept of a surface with bounded curvature and its curvature measure can be defined geometrically and without approximating Riemannian 2-manifolds \comp{pp. 6, 156}{AZ67}.
\subsubsection{Stability under convergence and completions}
This subsection provides results on surfaces with bounded curvature that will be useful for our investigation.\\
We start with the following convergence result:
\begin{trm}\label{trm_uni_surf_curv}\comp{pp. 88, 141, 240-241, 269}{AZ67} Let $\seq{X}{n}$ be a sequence of surfaces with bounded curvature such that the measures $\lvert\omega_{X_n}\rvert$ are uniformly bounded by $C$, $X$ be a length space and $p\in X$. Moreover let $\uni{X_n}{X}$ on some open neighborhood of $p$. We denote the topological embeddings corresponding to the uniform convergence by $\varphi_n$. Then, after passing to a subsequence, there is some $R>0$ such that the following statements apply:   
\begin{itemize}
\item[1)] $B_R(p)$ is isometric to an open subset of some surface with bounded curvature $Y$.  
\item[2)] The following inequalities hold:
\begin{itemize}
\item[a)] $\lvert\omega_Y\rvert\p{B_r(p)}\le \liminf\limits_{n\to \infty}\lvert\omega_{X_n}\rvert\p{ \varphi_n\p{B_r(p)}}$ for every $r\le R$. 
\item[b)] $ \mathcal{H}_Y^2\p{B_r(p)}\le \liminf\limits_{n\to \infty}\mathcal{H}_{X_n}^2\p{\varphi_n\p{B_r(p)}}$ for every $r\le R$.
\end{itemize}
\end{itemize}
\end{trm}
\noindent 
For our second result we need the following definition: Let $X$ be a metric space. Then we say that the \emph{area growth of open balls} in $X$ \emph{is at most quadratic} if there are $A,R>0$ such that
\begin{align*}
\mathcal{H}^2\p{B_r(x)}\le Ar^2\ \text{for every $x\in X$ and $r<R$}. 
\end{align*}
For a surface $S$ we denote the set of its interior points by $S^0$. Moreover a point $p$ in a surface with bounded curvature $X$ with $\theta_X(p)=0$ is called a \emph{peak point} of $X$.\\
In general the property of being a surface with bounded curvature does not carry over to the completion of the space. The upcoming result presents a positive example. Following Lytchak's arguments in \cite[pp. 630-631]{LW18} and \cite[pp. 7-8]{Lyt23}, we derive:      
\begin{trm}\label{trm_length_surf_curv}
Let $X$ be a length space that is homeomorphic to some closed surface and $F\subset X$ be a finite subset. Moreover let the area growth of open balls in $X$ be at most quadratic. Then the following statements apply: 
\begin{itemize}
\item[1)] $X\setminus F$ is a length space.
\item[2)] If $X\setminus F$ is a surface with bounded curvature such that the measure $\lvert \omega_{X\setminus F}\rvert$ is finite, then $X$ is also a surface with bounded curvature.  
\end{itemize}
\end{trm}
\begin{proof}  
For the sake of simplicity, we assume that $F=\cp{x}$.\\
1) We follow the arguments in \cite[pp. 630-631]{LW18}: First we choose some closed topological disc $D\subset X$ with $x\in D^0$. Moreover let $y\in J\coloneqq\partial D$.\\
For sufficiently small $r>0$ we have $S_r\coloneqq\partial S_r\subset D^0$ and that $S_r$ separates $x$ and $y$ in $D$. Using the bound on the area growth, the coarea formula \comp{p. 603}{LW18} implies that $\liminf\limits_{r\to 0}\mathcal{H}^1\p{S_r}=0$. Hence there is a sequence $\conv{r_n}{0}$ and simple closed curves $J_n\subset S_{r_n}$ with $\lim\limits_{n\to\infty}length\p{J_n}=0$ such that $J_n$ separates $x$ and $y$ in $D$ \comp{p. 623}{LW18}. In particular, $J_n$ bounds a closed topological disc $D_n$ with $x\in D_n^0$.\\
We conclude that $X\setminus F$ is a length space.\\        2) We follow the arguments in \cite[pp. 7-8]{Lyt23}: There is some $\varepsilon>0$ such that for every non-contractible loop $\gamma$ in $D\setminus\cp{p}$ of length at most $\varepsilon$ the intersection with $J$ is empty.\\
We may assume that $length\p{J_n}\le \frac{\varepsilon}{3}$ and choose a shortest non-contractible loop $\gamma_n$ in the topological cylinder bounded by $J$ and $J_n$.\\
Then $\gamma_n$ is a simple closed curve bounding a closed topological disc $A_n\subset D$ with $p\in A_n^0$. Moreover $A_n$ is convex and we have $\lim\limits_{n\to\infty} diam\p{A_n}=0$ and $\lim\limits_{n\to\infty} l_n\coloneqq length\p{\gamma_n}=0$. Due to the finiteness of the measure $\lvert \omega_{X\setminus F}\rvert$, after passing to a subsequence, $\gamma_n$ is free of peak points and the absolute rotation of $\gamma_n$, measured from the side of $X\setminus A_n$, \comp{pp. 272, 308}{AZ67} is bounded by some constant that does not depend on $n$.\\
We equip $X\setminus A_n^0$ with its induced length metric and glue this new space and the round hemisphere of length $l_n$ along some length-preserving homeomorphism between the boundary components. Then the gluing theorem in \cite[p. 289]{AZ67} implies that this gluing $X_n$ is a surface with bounded curvature such that the measure $\lvert \omega_{X_n}\rvert$ is bounded by some constant that does not depend on $n$.\\
By construction we have $\uni{X_n}{X}$. Hence the space $X$ is a surface with bounded curvature \comp{pp. 88, 141}{AZ67}.
\end{proof}
\noindent
For a closed Riemannian 2-manifold uniform constants as in the definition of the area growth can be calculated on the basis of its absolute curvature and diameter:
\begin{prop}\label{prop_Riem_area}\comp{p. 1773}{Shi99} Let $X$ be a closed Riemannian 2-manifold such that the measure $\lvert\omega_X\rvert$ is bounded by $C$. Then for every $x\in X$ the following inequality holds:
\begin{align*}
\mathcal{H}_X^2\p{B_r(x)}\le\p{\frac{2\pi+C}{2}}r^2\ \text{for every $r<2^{-1}diam(X)$.}
\end{align*}
\end{prop}
\subsubsection{Rotation of simple closed curves}
We want to introduce a notion of curvature for simple closed curves. Throughout this subsection $X$ denotes a surface with bounded curvature.\\ 
Let $\gamma_1,\gamma_2\subset X$ be geodesics emanating from the same point $p\in X$. Moreover we assume that the geodesics only intersect in $p$. In a sufficiently small open neighborhood of $p$ the geodesics bound two sectors. These sectors enable us to talk about a left- and a right-hand side of the "hinge" defined by the geodesics.\\
In particular, we can measure the so called \emph{sector angle} between $\gamma_1$ and $\gamma_2$ at $p$ from each side. In the following we write $\angle_l\p{\gamma_1,\gamma_2}$ and $\angle_r\p{\gamma_1,\gamma_2}$ for these two quantities. Furthermore we note that sector angles take values in $\left[0,\infty\right)$. The concept of sector angles gives rise to a geometric interpretation of the total angle of $p$: For every choice of the geodesics $\gamma_1$ and $\gamma_2$ the sum of the two corresponding sector angles is equal to $\theta_X(p)$. \comp{pp. 118-124}{AZ67}\\  
Let now $J\subset X$ be a two-sided simple closed curve. Then it makes sense to talk about a right- and a left-hand side of $J$.\\
The curve can be obtained as the uniform limit of piecewise geodesic simple closed curves that lie on the left-hand side of $J$. Each of them is a concatenation $\gamma_{n,1}\ast\ldots\ast\gamma_{n,k_n}$ of finitely many geodesics. The following quantity is well-defined and we call it the \emph{left rotation} of $J$:
\begin{align*}
\tau_l(J)\coloneqq \lim_{n\to\infty}\sum_{i=1}^{k_n}\pi-\angle_l\p{\gamma_{n,i},\gamma_{n,i+1}}  
\end{align*}
where $\gamma_{n,k_n+1}\coloneqq\gamma_{n,1}$. The \emph{right rotation} of $J$ is defined in an analogous way. \comp{pp. 191-192}{AZ67}\\
The Gauss-Bonnet theorem relates the geometry of a topological compact surface in $X$ to its topology:
\begin{trm}\comp{p. 192}{AZ67} Let $S\subset X$ be a compact topological surface. We denote its boundary components by $J_1,\ldots, J_n$. Then the following identity holds:
\begin{align*}
2\pi\chi(S)=\omega_X\p{S^0}+\sum_{k=1}^{n}\tau\p{J_k}\end{align*}
where the rotation is measured from the side of $S$.
\end{trm}
\noindent
By a \emph{topological cylinder}, we mean a metric space that is homeomorphic to $\mathbb{S}^1\times\left[0,1\right]$.\\   
Finally we state the following convergence result which is strongly based on a result by Burago in \cite[p. 14]{Bur67}:
\begin{trm}\label{trm_rot_lift} Let $X$ and $\seq{X}{n}$ be surfaces with bounded curvature such that the measures $\lvert\omega_{X_n}\rvert$ are uniformly bounded by $C$ and $\uni{X_n}{X}$ on some topological cylinder $Z\subset X$. We assume that there is a homeomorphism $\varphi_n\colon Z\to Z_n$ corresponding to the uniform convergence such that $\conv{d\p{\varphi_n,id_X}}{0}$. Moreover let $Z$ be free of peak points and $J\subset Z^0$ be a non-contractible simple closed curve. Then for every $\varepsilon>0$, after passing to a subsequence, the following statements apply:
\begin{itemize}
\item[1)] There is a non-contractible simple closed curve $J_\varepsilon\subset Z^0$ on the left-hand side of $J$ such that $\lvert\tau_l(J)-\tau_l\p{J_\varepsilon}\rvert\le\varepsilon$.
\item[2)] There are homeomorphisms $\psi_{\varepsilon,n}\colon Z\to Z_n$ with $\conv{d\p{\psi_{\varepsilon,n},\varphi_n}}{0}$ such that the following inequality holds: 
\begin{align*}
\lvert\tau_l\p{J_\varepsilon}-\lim_{n\to\infty}\tau_l\p{\psi_{\varepsilon,n}\circ J_\varepsilon}\rvert\le \lim_{n\to\infty}\lvert\omega_{X_n}\rvert\p{\psi_{\varepsilon,n}\circ J_\varepsilon}. \end{align*}   
\end{itemize}
\end{trm}
\begin{proof}
Using the definition of the rotation and a similar construction as in \cite[pp. 413-415]{Why35}, after passing to a subsequence, we obtain a piecewise geodesic curve $J_\varepsilon$ as in statement 1) and piecewise geodesic simple closed curves $J_{\varepsilon,n}\subset Z_n^0$ such that the number of edges is uniformly bounded and $\Haus{J_{\varepsilon,n}}{J_\varepsilon}$.\\
In the case that, after passing to a subsequence, the absolute left rotation of the curves $\varphi_n\p{J_\varepsilon}$ \comp{pp. 272, 308}{AZ67} is uniformly bounded statement 2) is covered by \cite[p. 14]{Bur67}.\\ Hence we only need to find homeomorphisms $\psi_{\varepsilon,n}$  corresponding to the uniform convergence such that this boundedness property is satisfied and $\conv{d\p{\psi_{\varepsilon,n},\varphi_n}}{0}$:\\
The length of the curves $J_{\varepsilon,n}$ is uniformly bounded. By \cite[p. 226]{AZ67} it follows that, after passing to a subsequence, the curves $\varphi_n^{-1}\circ J_{\varepsilon,n}$ converge uniformly to $J_\varepsilon$ as maps. A similar procedure as in \cite[pp. 227-229]{AZ67} yields that, after passing to a subsequence, there are homeomorphisms $\psi_{\varepsilon,n}\colon Z\to Z_n$ corresponding to the uniform convergence $\uni{X_n}{X}$ on $Z$ such that $\psi_{\varepsilon,n}\p{J_\varepsilon}=J_{\varepsilon,n}$ and $\conv{d\p{\psi_{\varepsilon,n},\varphi_n}}{0}$. Finally we note that the absolute left rotation of the curves $J_{\varepsilon,n}$ is uniformly bounded.  
\end{proof}
\section{Local uniform approximations}\label{sec_loc}
In this section we show Theorem \ref{trm_loc_uni_appr}.\\
For a closed surface $S$ and $\varepsilon>0$ we define $\mathcal{M}\p{S,\varepsilon}$ as the class of all length spaces $X$ that are homeomorphic to $S$ and satisfy the following property: Every loop in $X$ of diameter at most $\varepsilon$ is contractible.\\
The starting point of our investigation is the author's curvature-free formulation of Theorem \ref{trm_topo_struc}:
\begin{trm}\label{trm_topo_struc_free}\comp{p. 12}{Dot24} Let $\seq{X}{n}$ be a sequence in $\mathcal{M}\p{S,\varepsilon}$ and $\GH{X_n}{X}$. Then the following statements apply:
\begin{itemize}
\item[1)] If $S$ is homeomorphic to $\mathbb{S}^2$, then all maximal cyclic subset of $X$ are homeomorphic to $\mathbb{S}^2$.
\item[2)] If $S$ is not homeomorphic to $\mathbb{S}^2$, then one maximal cyclic subset of $X$ is homeomorphic to $S$ and all others are homeomorphic to $\mathbb{S}^2$. 
\end{itemize}
\end{trm}
\noindent 
Now we approximate closed topological discs in $X$ by closed topological discs in $X_n$:
\begin{lem}\label{lem_disc_approx}
Let $\seq{X}{n}$ be a sequence in $\mathcal{M}\p{S,\varepsilon}$ and $\GH{X_n}{X}$. Moreover let $D\subset X$ be a closed topological disc such that $Cut_X\cap D=\emptyset$ and $diam\p{\partial D}<\varepsilon$. Then, after passing to a subsequence, there are closed topological discs $D_n\subset X_n$ with $\Haus{D_n}{D}$ and $\Haus{\partial D_n}{\partial D}$.
\end{lem}
\begin{proof}
By \cite[p. 413]{Why35} there is a uniformly locally connected sequence of simple closed curves $J_n\subset X_n$ with $\Haus{J_n}{\partial D}$. Since $X_n\in\mathcal{M}\p{S,\varepsilon}$ and $diam(\partial D)<\varepsilon$, we further may assume that $J_n$ bounds a closed topological disc in $X_n$. We denote the closures of the two connected components of $X_n\setminus J_n$ by $A_{1,n}$ and $A_{2,n}$. After passing to a subsequence, we may assume the  sequences $A_{i,n}$ to be Hausdorff convergent.  By Lemma \ref{lem_lim_int} we may assume that $\Haus{A_{1,n}}{D}$.\\
For the sake of contradiction, we assume that $A_{1,n}$ is not a closed topological disc for infinitely many $n\inN$. Then $A_{2,n}$ is a closed topological disc. Therefore Lemma \ref{lem_lim_int} and Proposition \ref{prop_lim_disc} imply that the maximal cyclic subsets of $X$ are homeomorphic to $\mathbb{S}^2$. From Theorem \ref{trm_topo_struc_free} we derive that $X_n$ is homeomorphic to $\mathbb{S}^2$. It follows that $A_{1,n}$ is a closed topological disc. A contradiction.
\end{proof}
\noindent
The following definition goes back to Petersen \comp{p. 498}{Pet93}: A sequence of metric spaces $\seq{X}{n}$ is called \emph{uniformly locally 1-connected} if there is a non-decreasing map $p\colon \left[0,R\right)\to\left[0,\infty\right]$ satisfying the following properties:
\begin{itemize}
\item[1)] We have $p(0)=0$ and $p(t)\ge t$ for every $t\in \left[0,R\right)$. Moreover $p$ is continuous at $0$.  
\item[2)] For every $t\in\p{0,R}$ and $x\in X_n$ all loops in $B_t(x)$ are contractible in $B_{p(t)}(x)$.
\end{itemize}  
\begin{lem}\label{lem_uni_simply}
Let $\seq{X}{n}$ be a sequence in $\mathcal{M}\p{S,\varepsilon}$ and $\GH{X_n}{X}$. Moreover let $D\subset X$ with $Cut_X\cap D=\emptyset$ and $D_n\subset X_n$ be closed topological discs with $\Haus{D_n}{D}$ and $\Haus{\partial D_n}{\partial D}$. Then $\seq{D}{n}$ is uniformly locally 1-connected.     
\end{lem}
\begin{proof}
For the sake of contradiction, we assume that the sequence is not uniformly locally 1-connected. After passing to a subsequence, we then may assume the existence of $\delta>0$, $p\in D$ and simple closed curves $J_n\subset D_n$ such that $J_n$ is non-contractible in $U_\delta\p{J_n}\subset D_n$ and $\Haus{J_n}{\cp{p}}$. We note that $J_n$ bounds a closed topological disc in $X_n$. Moreover we denote the closures of the two connected components of $X_n\setminus J_n$ by $A_{1,n}$ and $A_{2,n}$.\\
We may assume that $A_{1,n}$ is a closed topological disc and that $\partial D_n\subset A_{2,n}$. Due to $\Haus{\partial D_n}{\partial D}$ and the fact that $J_n$ is non-contractible in $U_\delta\p{J_n}$ it follows that $diam\p{A_{i,n}\cap D_n}$ is uniformly bounded from below by a positive constant.\\
By Lemma \ref{lem_lim_int}, after passing to a subsequence, we find $A_i\subset X$ such that $\Haus{A_{i,n}}{A_i}$, $A_1\cup A_2=X$ and $A_1\cap A_2=\cp{p}$. Because $diam\p{A_i\cap D}>0$, it follows that $p\in Cut_X$. A contradiction.
\end{proof}
\noindent 
As a consequence of the aforementioned lemma, we derive the following corollary:
\begin{cor}\label{cor_uni_conv_disc}
Let $\seq{X}{n}$ be a sequence in $\mathcal{M}\p{S,\varepsilon}$ and $\GH{X_n}{X}$. Moreover let $D\subset X$ with $Cut_X\cap D=\emptyset$ and $D_n\subset X_n$ be closed topological discs with $\Haus{D_n}{D}$ and $\Haus{\partial D_n}{\partial D}$. Then for every closed topological disc $D'\subset D^0$ there are closed topological discs $D'_n\subset D_n^0$ such that $\uni{D'_n}{D'}$. Moreover the homeomorphisms $\varphi_n\colon D'\to D'_n$ corresponding to the uniform convergence can be chosen such that $\conv{d\p{\varphi_n,id_X}}{0}$.    
\end{cor}
\begin{proof}
First we double the chosen ambient space $Y$ in which $\Haus{D_n}{D}$ along the boundaries of the discs. For every $A\subset Y$ we write $2A$ for the corresponding doubled subset of this new ambient space.\\
Then we have $\Haus{2D_n}{2D}$ and the subsets are homeomorphic to $\mathbb{S}^2$. By Lemma \ref{lem_uni_simply} the sequence $\seq{D}{n}$ is uniformly locally 1-connected. Hence also $\seq{2D}{n}$ is uniformly locally 1-connected. In particular, we obtain that the same applies to $2D\cup\cp{2D_n}_{n\inN}$ \comp{p. 501}{Pet93}. From the last reference we also derive the existence of $\varepsilon_n$-equivalences $f_n\colon 2D\to 2D_n$ such that $\conv{\varepsilon_n}{0}$ and $\conv{d\p{f_n,id_{2X}}}{0}$.\\
Finally Theorem \ref{trm_alpha} closes the proof.  
\end{proof}
\noindent
Now Theorem \ref{trm_loc_uni_appr} follows from Theorem \ref{trm_topo_struc_free}, Lemma \ref{lem_disc_approx} and Corollary \ref{cor_uni_conv_disc}.\\
In a similar manner we obtain the following related result for cylinders:
\begin{prop}\label{prop_approx_cylinder}
Let $\seq{X}{n}$ be a sequence in $\mathcal{M}\p{S,\varepsilon}$ and $\GH{X_n}{X}$. Moreover let $Z\subset X$ be a topological cylinder satisfying the following properties:
\begin{itemize}
\item[1)] $X\setminus Z$ is disconnected.
\item[2)] $Cut_X\cap Z=\emptyset$.
\item[3)] The diameter of the boundary components of $Z$ is less than $\varepsilon$.
\end{itemize}
Then for every topological cylinder $Z'\subset Z^0$, after passing to a subsequence, there are topological cylinders $Z'_n\subset X_n$ such that $\uni{Z'_n}{Z'}$. Furthermore the homeomorphisms $\varphi_n\colon Z'\to Z'_n$ corresponding to the uniform convergence can be chosen such that $\conv{d\p{\varphi_n,id_X}}{0}$.
\end{prop}
\begin{proof}
The procedure is analogues to the previous proofs of this section. Only the following changes are made:\\  
We denote the boundary components of $Z$ by $J_1$ and $J_2$. Then there is a uniformly locally connected sequence of simple closed curves $J_{i,n}\subset X_n$ with $\Haus{J_{i,n}}{J_i}$. In particular, we may assume that $J_{i,n}$ bounds a closed topological disc in $X_n$.\\
We denote the subsurface of $X_n$ bounded by $J_1$ and $J_2$ by $Z_n$. Using the same arguments as in the proof of Lemma \ref{lem_disc_approx}, we conclude that, after passing to a subsequence, we have that $\Haus{Z_n}{Z}$ and $Z_n$ is a topological cylinder.\\
For the sake of contradiction, we assume that the sequence $\seq{Z}{n}$ is not uniformly locally 1-connected.\\
After passing to a subsequence, we find $\delta>0$, $p\in Z$, simple closed curves $\gamma_n\subset Z_n$ and compact subsurfaces $A_{1,n},A_{2,n}\subset X_n$ as in the proof of Lemma \ref{lem_uni_simply}. Especially, we may assume that $A_{1,n}$ contains some boundary component of $Z_n$. Depending on the contractibility of $\gamma_n$ in $Z_n$, the subsurface $A_{2,n}$ is either a closed topological disc in $Z_n$ or it also contains some boundary component of $Z_n$. Due to $\Haus{\partial Z_n}{\partial Z}$ and the fact that $\gamma_n$ is non-contractible in $U_\varepsilon\p{\gamma_n}$ it follows that $diam\p{A_{i,n}\cap Z_n}$ is uniformly bounded from below by a positive constant.\\
The contradiction follows exactly as in the aforementioned proof.\\
Along the same lines as in the proof of Corollary \ref{cor_uni_conv_disc}, we find the desired topological cylinders and homeomorphisms.
\end{proof}
\section{Description of the limit spaces}\label{sec_main}
This section is devoted to the proof that all limit spaces meet the description in Theorem \ref{trm_main}.\\
We define $\mathcal{M}\p{S,\varepsilon,C}$ as the class of all surfaces with bounded curvature $X$ in  $\mathcal{M}\p{S,\varepsilon}$ such that the measure $\lvert\omega_X\rvert$ is bounded by $C$.
\subsection{Bounded curvature of maximal cyclic subsets}
As a first step, we show that the maximal cyclic subsets of the limit space are surfaces with bounded curvature.\\
In the following we write $X^{int}$ for a metric space $X$ equipped with its induced length metric.\\
By Theorem \ref{trm_topo_struc}, Theorem \ref{trm_uni_surf_curv} and Proposition \ref{prop_approx_cylinder} we have the following result:
\begin{cor}\label{cor_curv_est}
Let $\seq{X}{n}$ be a sequence in $\mathcal{M}\p{S,\varepsilon,C}$, $\GH{\p{X_n,\lvert\omega_{X_n}\rvert}}{\p{X,\lvert\omega\rvert_\infty}}$
and $\GH{\p{X_n,\mathcal{H}^2_{X_n}}}{\p{X,\mathcal{H}^2_\infty}}$. Moreover let $T\subset X$ be a maximal cyclic subset and $U\coloneqq T\setminus Cut_X$. Then the following statements apply: \begin{itemize}
\item[1)] $U^{int}$ is a surface with bounded curvature and $\lvert\omega_{U^{int}}\rvert(V)\le \lvert\omega\rvert_\infty(V)$ for every connected open subset $V\subset U^{int}$. 
\item[2)] $\mathcal{H}^2_X(V)\le\mathcal{H}^2_\infty(V)$ for every connected open subset $V\subset T$. 
\end{itemize} 
\end{cor}
\noindent
From the aforementioned corollary, Proposition \ref{prop_Riem_area} and the fact that $X_n$ can be obtained as the uniform limit of spaces in $\mathcal{R}\p{S,2C}$ \comp{p. 144}{Res93} we derive a bound on the area growth of open balls:
\begin{prop}\label{prop_area_growth}
Let $\seq{X}{n}$ be a sequence in $\mathcal{M}\p{S,\varepsilon,C}$ and $\GH{X_n}{X}$. Then the area growth of open balls in $X$ is at most quadratic. 
\end{prop}
\noindent
Due to Theorem \ref{trm_topo_struc} every maximal cyclic subset contains only finitely many cut points of $X$. Using Theorem \ref{trm_length_surf_curv} and the last two results, we conclude that the maximal cyclic subsets are surfaces with bounded curvature. Moreover we obtain a bound on the total absolute curvature of these surfaces.
\begin{cor}\label{cor_surf_formula}
Let $\seq{X}{n}$ be a sequence in $\mathcal{M}\p{S,\varepsilon,C}$ and $\GH{\p{X_n,\lvert\omega_{X_n}\rvert}}{\p{X,\lvert\omega\rvert_\infty}}$. Then the following statements apply:
\begin{itemize}
\item[1)] Every maximal cyclic subset of $X$ is a surface with bounded curvature.
\item[2)] Let $\seq{T}{n}$ be an enumeration of the maximal cyclic subsets of $X$. Then the following inequality holds:
\begin{align*}
\sum\limits_{n=1}^{\infty}\lvert\omega_{T_n}\rvert\p{T_n\setminus Cut_X}\le \lvert\omega \rvert_\infty\p{X\setminus Sing_X}. \end{align*}  
\end{itemize}
\end{cor}
\subsection{Geometry of singular points}
In this subsection we study the geometry of the singular points in the limit space. More precisely, we show the following inequality:
\begin{prop}\label{prop_sing_formula}
Let $\seq{X}{n}$ be a sequence in $\mathcal{M}\p{S,\varepsilon,C}$ and $\GH{\p{X_n,\lvert\omega_{X_n}\rvert}}{\p{X,\lvert\omega\rvert_\infty}}$. Moreover let $p\in Sing_X$ and $\seq{T}{n}$ be an enumeration of the maximal cyclic subsets of $X$. Then the following inequality holds:
\begin{align*}
2\pi\lvert ind(p)-2\rvert+\sum_{n=1}^{\infty}\mathbbold{1}_{T_n}(p)\theta_{T_n}(p)\le \lvert\omega\rvert_\infty(p).\end{align*}
\end{prop}
\noindent
If $p$ is an endpoint, then we have $ind(p)=1$ and $p$ lies in no maximal cyclic subset. Hence the inequality becomes $2\pi\le \lvert\omega\rvert_\infty(p)$.\\For cut points we have $ind(p)\ge 2$ and $p$ lies in at most finitely many maximal cyclic subsets $T_{n_1},\ldots,T_{n_k}$. In this case the inequality becomes
\begin{align*}
2\pi \p{ind(p)-2}+\sum_{i=1}^{k}\theta_{T_{n_i}}(p)\le \lvert\omega\rvert_\infty(p).    
\end{align*}
We will conclude the aforementioned proposition from the following lemma:
\begin{lem}\label{lem_approx_sing_surf}
Let $\seq{X}{n}$ be a sequence in $\mathcal{M}\p{S,\varepsilon,C}$ and $\GH{\p{X_n,\lvert\omega_{X_n}\rvert}}{\p{X,\lvert\omega\rvert_\infty}}$. Moreover let $p\in Sing_X$ and $p_n\in X_n$ with $\conv{p_n}{p}$. We denote the maximal cyclic subsets containing $p$ by $T_1,\ldots,T_N$. Then for every $\delta>0$, after passing to a subsequence, there are compact subsurfaces $S_n\subset X_n$ satisfying the following properties:
\begin{itemize}
\item[1)] We have $p_n\in S_n^0$ and $\Haus{S_n}{\cp{p}}$.
\item[2)]  $S_n$ has $ind(p)$ boundary components $J_{1,n},\ldots, J_{ind(p),n}$ and $\chi\p{S_n}=2-ind(p)$.
\item[3)] If $p$ is a cut point, then the following estimates apply:
\begin{itemize}
\item[a)] $\tau\p{J_{i,n}}\ge \theta_{T_i}(p)-2\delta$ for all $1\le i\le N$
\item[b)] $\tau\p{J_{i,n}}\ge -2\delta$ for all $N+1\le i\le ind(p)$ 
\end{itemize}
where the rotation is measured from the side of $p_n$.
\item[4)] If $p$ is an endpoint, then $\tau\p{J_{1,n}}\le 2\delta$ where the rotation is measured from the side of $p_n$.
\end{itemize} 
\end{lem}
\noindent
First we explain how to derive the proposition from the lemma:\\
We start with the case that $p$ is a cut point. Since $\Haus{S_n}{\cp{p}}$, after passing to a subsequence, we may assume that 
\begin{align*}
\lvert\omega_{X_n}\rvert\p{S_n}\le\lvert\omega\rvert_\infty(p)+\delta.    
\end{align*}
From the identity for $\chi\p{S_n}$ and
the Gauss-Bonnet theorem we derive
\begin{align*}
\sum_{i=1}^{ind(p)}\tau\p{J_{i,n}}\le2\pi\p{2-ind(p)}+\lvert\omega\rvert_\infty(p)+\delta   
\end{align*}
where the rotation is measured from the side of $p_n$. Due to the estimates for the rotation of the boundary components we derive
\begin{align*}
2\pi\p{ind(p)-2}+\sum_{i=1}^{N}\theta_{T_i}(p)\le \lvert\omega\rvert_\infty(p)+\p{2\cdot ind(p)+1}\delta. 
\end{align*}
Taking the limit $\conv{\delta}{0}$, we get the desired inequality.\\
If $p$ is an endpoint, then we have $ind(p)=1$ and the Gauss-Bonnet theorem implies 
\begin{align*}
2\pi\le \tau\p{J_{1,n}}+\lvert\omega\rvert_\infty(p)+\delta. 
\end{align*} 
By the estimate for the rotation of the boundary component we get 
\begin{align*}
2\pi\le\lvert\omega\rvert_\infty(p)+3\delta.     
\end{align*}
As before, taking the limit $\conv{\delta}{0}$, we obtain the desired inequality.\\
Now we turn to the proof of the lemma:
\subsubsection{Approximation of curves}
First we prove that certain simple closed curves $J$ in the limit space admit Hausdorff approximations by simple closed curves whose rotations converge to the rotation of $J$: 
\begin{lem}\label{lem_approx_curve}
Let $\seq{X}{n}$ be a sequence in $\mathcal{M}\p{S,\varepsilon,C}$ and $\GH{\p{X_n,\lvert\omega_{X_n}\rvert}}{\p{X,\lvert\omega\rvert_\infty}}$. Moreover let $D\subset X$ be an open topological disc, $p\in D$ with $Cut_X\cap \p{D\setminus\cp{p}}=\emptyset$ and $p_n\in X_n$ with $\conv{p_n}{p}$. Then for every $\delta>0$, after passing to a subsequence, the following statements apply:
\begin{itemize}
\item[1)] There is a simple closed curve $J\subset D$ such that $J$ bounds a closed topological disc in $D$ containing $p$. Moreover we have $\tau(J)\ge-\delta$ where the rotation is measured from the side of $p$.
\item[2)] There are separating simple closed curves $J_n\subset X_n$ with $p_n\not\in J_n$, $\Haus{J_n}{J}$, $\lvert \omega_{X_n}\rvert\p{J_n}\le\delta$ and $\tau\p{J_n}\ge \tau(J)-\delta$ where the rotation is measured from the side of $p_n$ and $p$ respectively.
\end{itemize} 
\begin{proof}
There is some $R\in\left(0,\frac{\varepsilon}{2}\right]$ such that for all $0<r<R$ the ball $\bar{B}_r(p)$ is a closed topological disc in $D$ whose surface boundary coincides with $\partial B_r(p)$ and the rotation of the boundary, measured from the side of $p$, is larger than $-\delta$ \comp{p. 151-152}{Res93}. Since $\lvert\omega\rvert_\infty$ is a finite measure, there is some $r_0\in\left(0,R\right)$ such that $J\coloneqq \partial B_{r_0}(p)$ satisfies $\lvert\omega\rvert_\infty(J)=0$.\\
Let $\tilde{\delta}>0$ with $r_0\pm\Tilde{\delta}\in\left(0,R\right)$. Then the curves $\partial B_{r_0\pm\Tilde{\delta}}(p)$ bound a topological cylinder $Z_{\Tilde{\delta}}\subset D$ such that $J$ is a non-contractible curve in its interior. From Theorem \ref{trm_topo_struc_free} follows that $Z_{\Tilde{\delta}}$ separates $X$. Because  $\lim\limits_{\tilde{\delta}\to 0}\lvert\omega\rvert_\infty\p{Z_{\Tilde{\delta}}}=\lvert\omega\rvert_\infty(J)$, we may assume that $\lvert\omega\rvert_\infty\p{Z_{\Tilde{\delta}}}<\min\cp{2\pi,\delta}$. In particular, it follows that $Z_{\Tilde{\delta}}$ is free of peak points.\\
By Proposition \ref{prop_approx_cylinder}, after passing to a subsequence, we may assume the existence of topological cylinders $Z_{\Tilde{\delta},n}\subset X_n$ such that $\uni{Z_{\Tilde{\delta},n}}{Z_{\Tilde{\delta}}}$. Moreover the homeomorphisms $\varphi_n\colon Z_{\Tilde{\delta}}\to Z_{\Tilde{\delta},n}$ corresponding to the uniform convergence can be chosen such that $\conv{d\p{\varphi_n,id_X}}{0}$. Since $\lvert\omega\rvert_{\infty}\p{Z_{\tilde{\delta}}}<\delta$, after passing to subsequence, we may assume that $\lvert\omega_{X_n}\rvert\p{Z_{\tilde{\delta},n}}\le \delta$.\\ 
We define $J_{\Tilde{\delta},n}\coloneqq\varphi_n\circ J$. Then we have $\Haus{J_{\Tilde{\delta},n}}{J}$, $\lvert\omega_{X_n}\rvert\p{J_{\Tilde{\delta},n}}\le\delta$ and we may assume that $p_n\not\in J_{\Tilde{\delta},n}$. Since $diam (J)<\varepsilon$, we further may assume that $J_{\Tilde{\delta},n}$ is separating in $X_n$. By Lemma \ref{lem_lim_int} and  Theorem \ref{trm_rot_lift}, after passing to a subsequence, we may assume that
 \begin{align*}
 \lvert\tau(J)-\lim_{n\to\infty}\tau\p{J_{\Tilde{\delta},n}}\rvert\le\lvert \omega\rvert_\infty\p{Z_{\Tilde{\delta}}}\end{align*}
where the rotation is measured from the side of $p_n$ and $p$ respectively. Hence for sufficiently small $\tilde{\delta}$, after passing to a subsequence, we have $\tau\p{J_{\tilde{\delta},n}}\ge \tau(J)-\delta$.      
\end{proof}
\end{lem}
\subsubsection{Approximation of points }
As a next step, we approximate certain points of the limit space by simple closed curves such that the rotation is uniformly bounded from below: 
\begin{lem}
Let $\seq{X}{n}$ be a sequence in $\mathcal{M}\p{S,\varepsilon,C}$ and $\GH{\p{X_n,\lvert\omega_{X_n}\rvert}}{\p{X,\lvert\omega\rvert_\infty}}$. Moreover let $I\subset X$ be an open subset that is isometric to some open interval. Then, after passing to a subsequence, there are topological cylinders $Z_n\subset X_n$ and $J\subset I$ such that $\Haus{Z_n}{J}$ and $\Haus{\partial Z_n}{\partial J}$.    
\end{lem}
\begin{proof}
We may assume that the midpoint $p$ of the interval satisfies $\lvert\omega\rvert_\infty(p)=0$. Moreover we choose $0<r\le \frac{\varepsilon}{2}$ such that $R\coloneqq 100r<2^{-1}length(I)$.\\
Let $p_n\in X_n$ such that $\conv{p_n}{p}$. By \cite[p. 77, 144]{Res93} and \cite[p. 1772]{Shi99} we may assume that $X_n$ is a Riemannian 2-manifold, $r$ and $R$ are non-exceptional with respect to $p_n$ in the sense of the latter reference and $\bar{B}_r\p{p_n}$ is a compact surface whose surface boundary is given by $\partial B_r\p{p_n}$. We write $Z_n$ for the union of $\bar{B}_r\p{p_n}$ with all closed topological discs in $B_R\p{p_n}\setminus B_r\p{p_n}$ that are bounded by some connected component of $\partial B_r\p{p_n}$.\\ Choosing $r$ small enough, after passing to a subsequence, we may assume the inequality $\vert\omega_{X_n}\rvert\p{B_R\p{p_n}}\le 1$. Then $Z_n$ is a compact surface whose boundary is given by connected components of $\partial B_r\p{p_n}$ and, choosing $r$ even smaller, the last reference implies that $\chi\p{Z_n}\ge 0$. Since $X\setminus \bar{B}_R(p)$ consists of two connected components, Lemma \ref{lem_lim_int} implies that, after passing to a subsequence, $Z_n$ has at least two boundary components. As a consequence, $Z_n$ is a topological cylinder.\\
Finally, after passing to a subsequence, there is some $J\subset I$ such that $\Haus{Z_n}{J}$ and $\Haus{\partial Z_n}{\partial J}$. \end{proof}
\noindent
We define the \emph{length} of a topological cylinder $Z$ as the distance between its boundary components and denote it by $length(Z)$.
\begin{lem}
Let $X_n$ be closed surfaces with bounded curvature and $Z_n\subset X_n$ be topological cylinders. Moreover we assume that $\seq{Z}{n}$ Gromov-Hausdorff converges to some compact interval $\left[0,l\right]\subset\mathbb{R}$ and $\Haus{\partial Z_n}{\cp{0,l}}$. Then for every $\delta>0$, after passing to a subsequence, there are topological cylinders $Z'_n\subset Z_n$ satisfying the following properties:
\begin{itemize}
\item[1)] The boundary components of $Z_n'$ are piecewise geodesic.
\item[2)] The boundary components of $Z'_n$ are non-contractible in $Z_n$. 
\item[3)]  We have $length\p{Z'_n}\ge \frac{l}{2}-2\delta$ and $d\p{\partial Z'_n,\partial Z_n}\ge\frac{l}{4}-2\delta$. 
\item[4)] The length of the boundary components of $Z'_n$ is at most $\delta$. 
\end{itemize}
\end{lem}
\begin{proof}
Let $\gamma\subset Z_n$ be a geodesic with $length(\gamma)=length\p{Z_n}$ that connects the boundary components of $Z_n$. We denote the midpoint of $\gamma$ by $p$ and its endpoints by $x_1$ and $x_2$.\\
Due to the convergence of the cylinders, after passing to a subsequence, we obtain that $length(\gamma)\ge l-\delta$, $S\coloneqq\partial B_{\frac{l}{4}}\p{p}\subset Z_n^0$ and that $S$ separates $x_1$ and $x_2$ in $Z_n$. Moreover we may assume the existence of some simple closed curve $J_1\subset S$ with $diam\p{J_1}\le\frac{\delta}{3}$ that separates $x_1$ and $x_2$ in $Z_n$ \comp{p. 623}{LW18}. In particular, $J_1$ is non-contractible in $Z_n$ and we may assume that $J_1$ separates $x_1$ and $p$ in $Z_n$.\\ Using a similar argument, we also may assume that there is a second simple closed curve $J_2\subset S$ with $J_1\cap J_2=\emptyset$ and $diam\p{J_2}\le\frac{\delta}{3}$ that is non-contractible in $Z_n$ and separates $x_2$ and $p$ in $Z_n$.\\
There is some $\varepsilon_n>0$ such that every loop in $Z_n$ of diameter at most $\varepsilon_n$ is contractible in $Z_n$. Let $y$ be the intersection point of $J_i$ with $\gamma$ and $0<\Tilde{\delta}\le\min\cp{\frac{\varepsilon_n}{2},\frac{\delta}{3}}$. We may assume that $length\p{J_i}>\tilde
{\delta}$ and choose a subdivision $t_0\coloneqq 0<t_1<\ldots<t_k\coloneqq 1$ of the domain of $J_i$ such that $length\p{J_i|_{\left[t_j,t_{j+1}\right]}}\le\tilde{\delta}$. After passing to a subsequence, there is some geodesic $\alpha_j\subset Z$ connecting $y$ with $J_i\p{t_j}$ and we may assume that $\alpha_j\cap\gamma$ and $\alpha_j\cap\alpha_{j+1}$ are connected.\\
By construction there is some $t_0<t_j<t_k$ such that $\alpha_j$ is homotopic to $J_i|_{\left[0,t_j\right]}$ and $\alpha_{j+1}$ is homotopic to $J_i|_{\left[t_{j+1},1\right]}$. After passing to a subsequence, there is some geodesic $\xi\subset Z$ connecting $J_i\p{t_j}$ with $J_i\p{t_{j+1}}$ such that $\xi\cap \alpha_j$ and $\xi\cap \alpha_{j+1}$ are connected. In particular, $\xi$ is homotopic to $J_i|_{\left[t_j,t_{j+1}\right]}$. We derive that the concatenation $\alpha_j\ast \xi\ast \alpha_{j+1}$ is homotopic to $J_i$.\\ Up to parametrization there is exactly one simple closed curve $\beta_i$ contained in this concatenation. The curve $\beta_i$ is non-contractible in $Z_n$ and piecewise geodesic. After passing to a subsequence, we have $\beta_1\cap\beta_2=\emptyset$ and the curves bound a topological cylinder $Z_n'\subset Z_n$. In particular, we get $length\p{Z'_n}\ge \frac{l}{2}-2\delta$, $d\p{\partial Z_n',\partial Z_n}\ge\frac{l}{4}-2\delta$ and the length of the boundary components of $Z_n'$ is at most $\delta$.
\end{proof}
\noindent
Next we show a general result about surfaces with bounded curvature. The exact estimates in the statement do not play an essential role in the further discussion. It is only important that $\tau(J)$ can be chosen close to zero if $\frac{c}{a}$ and $\lvert\omega_X\rvert(Z)$ are close to zero.    
\begin{lem}
Let $X$ be a closed surface with bounded curvature and $Z'\subset Z\subset X$ be topological cylinders with piecewise geodesic boundary components. Moreover we assume that the boundary components of $Z'$ are non-contractible in $Z$ and that $\delta\coloneqq\lvert\omega_X\rvert(Z)<\frac{\pi}{5}$. We introduce the following notation: \begin{itemize}
\item $a\coloneqq length\p{Z'}$
\item $b\coloneqq d\p{\partial Z',\partial Z}$
\item $c\coloneqq\text{The length of a longest boundary component of $Z'$.}$ 
\item $\alpha\coloneqq\frac{c}{a}+7\delta$ 
\end{itemize} 
If $c<b$ and $\arccos(\alpha)\ge\frac{\pi}{2}-\delta$, then there is some simple closed curve $J\subset Z$ such that $\lvert\tau(J)\rvert\le 4\delta$. The estimate does not depend on the direction from which the rotation is measured. 
\end{lem}
\begin{proof}
First we choose a geodesic $\gamma\subset Z$ with $length(\gamma)=length(Z)$ that connects the boundary components of $Z$. Then we cut $Z^{int}$ along $\gamma$. The resulting space $Q$ is a surface with bounded curvature in the sense of \cite[p. 141]{Res93} that is homeomorphic to the closed disc. Let $\pi\colon Q\to Z^{int}$ be the projection map corresponding to the cutting and $\gamma_1$ and $\gamma_2$ be the two connected components of $\pi^{-1}(\gamma)$. We note that these curves are geodesics in $Q$.\\
Let $x_1,x_2\in \gamma_2$ with $x_1\neq x_2$. Moreover let $\xi_i\subset Q$ be a geodesic from $x_i$ to some closest point of $\gamma_1$ such that $\xi_i$ does not intersect $\pi^{-1}\p{\partial Z}$. We further assume that the subsets $\xi_i\cap \gamma_2$ and $\xi_1\cap \xi_2$ are connected. For the sake of simplicity, we only consider the case that $\xi_i\cap \gamma_2$ is a singleton and $\xi_1\cap \xi_2$ is empty. All remaining cases can be obtained similarly or follow from this case.\\ The first variation formula \comp{pp. 1034, 1036}{Lyt05} shows that the two sector angles between $\gamma_1$ and $\xi_i$ are at least $\frac{\pi}{2}$ and it follows that they lie in $\left[\frac{\pi}{2},\frac{\pi}{2}+\delta\right]$. We note that the sum of the two sector angles between $\xi_i$ and $\gamma_2$ lies in $\left[\pi,\pi+\delta\right]$.\\
There is a unique closed topological disc $R\subset Q$ that is bounded by a subset of $\gamma_1\cup \xi_1\cup \gamma_2\cup\xi_2$. The Gauss-Bonnet theorem implies that the sum of the sector angle between $\xi_1$ and $\gamma_2$ and the sector angle between $\xi_2$ and $\gamma_2$, both measured from the side of $R$, lies in $\left[\pi-3\delta,\pi+\delta\right]$. We conclude that the sector angle between $\xi_1$ and $\gamma_2$, measured from the side of $R$, differs from the sector angle between $\xi_2$ and $\gamma_2$, measured from outside of $R$, by at most $5\delta$.\\
We may assume that $\gamma_2$ is parametrized by arc length and denote its domain by $I$. Furthermore let $I_0\subset I$ be the maximal subinterval such that
$\pi^{-1}\p{\gamma\cap Z'}\subset \gamma_2\p{I_0}$ and the following property is satisfied: For every $t\in I_0$ and every geodesic connecting $\gamma_2(t)$ with $\gamma_1$ of length $d_{\gamma_1}\p{\gamma_2(t)}\coloneqq d\p{\gamma_1,\gamma_2(t)}$ the intersection with $\pi^{-1}\p{\partial Z}$ is empty.\\
Using the aforementioned estimate for the difference of the sector angles, the first variation formula yields
\begin{align*}
\lvert \partial^+d_{\gamma_1}\p{\gamma_2(s)}-\partial^+d_{\gamma_1}\p{\gamma_2(t)}\rvert\le 7\delta   
\end{align*}
for all $s,t\in I_0$.\\
For the sake of contradiction, we assume that $\lvert \partial^+d_{\gamma_1}\p{\gamma_2(t)}\rvert > \alpha$ for all $t\in I_0$. Then we may assume that $\partial^+d_{\gamma_1}\p{\gamma_2(t)}>\frac{c}{a}$. By the mean value theorem for one-sided differentiable functions in \cite[p. 473]{MV86} we have
\begin{align*}
d_{\gamma_1}\p{\gamma_2(s)}>\frac{c}{a}\p{s-t}    
\end{align*}
for all $s\ge t$. There is a choice of the parameters such that $d_{\gamma_1}(s)\le c$ and $s-t\ge a$. This yields $c>2c$. A contradiction.\\
It follows the existence of some $t_0\in I_0$ with $\lvert \partial^+d_{\gamma_1}\p{\gamma_2\p{t_0}}\rvert \le\alpha$. Due to the first variation formula there is a geodesic $\xi\subset Q$ connecting $\gamma_2\p{t_0}$ with $\gamma_1$ of length $d_{\gamma_1}\p{\gamma_2\p{t_0}}$ such that $\xi\cap \gamma_2$ is a singleton and the sector angles between $\xi$ and $\gamma_2$ lie in  $\left[\arccos(\alpha),\pi-\arccos(\alpha)\right]$.\\
We connect the endpoints of $\pi\p{\xi}$ by some subarc of $\gamma$ and denote the constructed simple closed curve in $Z$ by $J$. Removing the endpoints of $\pi(\xi)$ from $J$, yields two local geodesics. By the estimates for the sector angles we finally derive
\begin{align*}
\lvert\tau(J)\rvert\le \frac{\pi}{2}-\arccos(\alpha)+3\delta\le 4\delta. 
\end{align*}
In particular, the estimate does not depend on the direction from which the rotation is measured.
\end{proof}
\noindent
Since $\lvert\omega\rvert_\infty$ is a finite measure, every open subset contains some point with measure zero. Combining the last three results, we therefore obtain the following corollary: 
\begin{cor}\label{cor_approx_point}
Let $\seq{X}{n}$ be a sequence in $\mathcal{M}\p{S,\varepsilon,C}$ and $\GH{X_n}{X}$. Moreover let $I\subset X$ be an open subset that is isometric to some open interval. Then for every $\delta>0$, after passing to a subsequence, there is a point $p\in I$ and  simple closed curves $J_n\subset X_n$ such that $\Haus{J_n}{\cp{p}}$ and $\lvert\tau\p{J_n}\rvert\le \delta$. The estimate does not depend on the direction from which the rotation is measured.   
\end{cor}
\noindent
Finally we are able to show Lemma \ref{lem_approx_sing_surf}. As already discussed, this is the last step to prove Proposition \ref{prop_sing_formula}.
\begin{proof}[Proof of Lemma \ref{lem_approx_sing_surf}]
We write $A_1,\ldots,A_{ind(p)}$ for the connected components of $B_r(p)\setminus\cp{p}$.\\
By Lemma \ref{lem_approx_curve} and Corollary \ref{cor_approx_point}, after passing to a subsequence, we find disjoint simple closed curves $J_{i,n}$ bounding closed topological discs in $X_n$ such that $J_{i,n}$ Hausdorff converges to some subset $J_i\subset A_i$ and $\tau\p{J_{i,n}}\ge-2\delta$ where the rotation is measured from the side of $p_n$.\\
For $1\le i\le N$ we may sharpen the construction further: By Lemma \ref{lem_approx_curve} we may assume that the subset $J_i$ is the boundary of a closed topological disc $D_i\subset A_i\cup\cp{p}$ with $p\in D_i^0$ and $\tau\p{J_{i,n}}\ge\tau\p{J_i}-\delta$ where the rotation is measured from the side of $p_n$ and $p$ respectively. Choosing $r$ small enough, we may assume that $\lvert\omega_{T_i}\rvert\p{D_i\setminus\cp{p}}\le\delta$. Then the Gauss-Bonnet theorem yields
$\tau\p{J_i}\ge\theta_{T_i}(p)-\delta$ and hence $\tau\p{J_{i,n}}\ge\theta_{T_i}(p)-2\delta$.\\
Moreover we may assume that all curves $J_{k,n}$ with $k\neq i$ lie on the side of $J_{i,n}$ that contains $p_n$. This yields that the curves bound a compact subsurface $S_n\subset X_n$ with $\partial S_n=\cup_{i=1}^{ind(p)}J_{i,n}$ and $p_n\in S_n^0$.\\
In the following we denote the simple closed curves constructed via the radius $r$ by $J_{i,n}(r)$ and the corresponding compact subsurfaces by $S_n(r)$. Choosing a diagonal sequence, we may assume that the subsets $S_n\p{\frac{1}{n}}$ Hausdorff converge to $\cp{p}$. Moreover due to Theorem \ref{trm_topo_struc_free} we may assume that $\chi\p{S_n}=2-ind(p)$.\\
This gives the construction in the case that $p$ is a cut point.\\ 
The proof for endpoints proceeds similar: For arbitrarily small $r$ the subset $A_1$ is isometric to some open interval or $A_1$ contains  infinitely many maximal cyclic subsets of $X$.\\
We start with the first case: From Corollary \ref{cor_approx_point}, after passing to a subsequence, we obtain simple closed curves $J_{1,n}\subset X_n$ bounding closed topological discs such that
$J_{1,n}$ Hausdorff converges to some point in $A_1$ and $\tau\p{J_{1,n}}\le 2\delta$ where the rotation is measured from the side of $p_n$.\\
In the second case by Lemma \ref{lem_approx_curve} we may assume the existence of such curves such that $J_{1,n}$ Hausdorff converges to the boundary of a closed topological disc in $A_1$. In addition we may assume that $\omega_{X_n}\p{J_{1,n}}\le\delta$ and $\tau\p{J_{1,n}}\ge -\delta$ where the rotation is measured from the side opposite to $p_n$. Then it follows that the rotation measured from the side of $p_n$ is at most $2\delta$.\\ 
The rest of the proof proceeds as before.
\end{proof}
\noindent
Because the measure $\lvert\omega\rvert_{\infty}$ is bounded by $C$, Theorem \ref{trm_main} is a direct consequence of Corollary \ref{cor_surf_formula} and Proposition \ref{prop_sing_formula}.
\section{A converse of the main result}
This section is devoted to the proof of Theorem \ref{trm_converse}.\\
First we note that all spaces $X\in\mathcal{L}\p{S,C}$ are semi-locally 1-connected \comp{pp. 1-2}{Dot24}.\\
A deeper analysis of the construction in \cite[p. 9]{Dot24} proves our first result. We give a brief sketch of the construction:\\
Let $\seq{T}{n}$ be an enumeration of the maximal cyclic subsets in $X$. We shrink each connected component of $\cup_{i=n+1}^{n+k}T_i$ to a point and denote the constructed space by $X_{n,k}$. After passing to a subsequence, the spaces $X_{n,k}$ Gromov-Hausdorff converge to some space $X_n$. The maximal cyclic subsets of $X_n$ are isometric to the first $n$ maximal cyclic subsets of $X$ and we have $\GH{X_n}{X}$.\\ 
From the construction we deduce the following lemma:
\begin{lem}
Let $X\in\mathcal{L}\p{S,C}$. Then $X$ can be obtained as the Gromov-Hausdorff limit of a uniformly semi-locally 1-connected sequence $X_n\in\mathcal{L}\p{S,C}$ such that $X_n$ has only finitely many maximal cyclic subsets.
\end{lem} 
\noindent 
In the following we write $L_{k,n}$ for the metric space obtained by gluing $k$ disjoint copies of $\left[0,2^{-n}\right]$ along $0$.\\ 
A space $X\in \mathcal{L}\p{S,C}$ having only finitely many maximal cyclic subsets is a successive metric wedge sum of its maximal cyclic subsets and finitely many compact intervals. Replacing every wedge point $p\in X$ with $L_{ind(p),n}$ and then slightly moving the intervals, we derive the following result: 
\begin{lem}\label{lem_approx_finite}
Let $X$ be a space in $\mathcal{L}\p{S,C}$ having only finitely many maximal cyclic subsets. Then $X$ can be obtained as the Gromov-Hausdorff limit of spaces $X_n\in \mathcal{L}\p{S,C}$ satisfying the following properties:
\begin{itemize}
\item[1)] $X_n$ is a successive metic wedge sum of the maximal cyclic subsets of $X$ and finitely many compact intervals.
\item[2)] Every wedge point is the endpoint of some of the intervals. Moreover it lies in some maximal cyclic subset of $X_n$ or the interior of some of the intervals.
\item[3)] Every wedge point lies in exactly two of the wedged spaces.
\item[4)] The maximal cyclic subsets of $X_n$ do not intersect. 
\end{itemize}
\noindent
The upcoming proposition provides a "tool" to simplify the approximating spaces further.
\begin{prop}\label{prop_approx_wedge_sum}
Let $S_1$ and $S_2$ be closed surfaces with bounded curvature and $X$ be a metric wedge sum of them. We further assume that every loop in $S_1$ of diameter at most $\varepsilon$ is contractible in $S_1$ and that $S_2$ is homeomorphic to $\mathbb{S}^2$. Moreover let $F\coloneqq\cp{p_1,\ldots,p_k}\subset X\setminus\cp{p}$ where $p$ denotes the wedge point of $X$. Then $X$ can be obtained as the Gromov-Hausdorff limit of surfaces with bounded curvature $X_n$ satisfying the following properties:
\begin{itemize}
\item[1)] $X_n$ is homeomorphic to $S_1$. 
\item[2)] We have
\begin{align*}
\lim\limits_{n\to\infty} \lvert\omega_{X_n}\rvert\p{X_n}\le\sum_{i=1}^{2}\lvert\omega_{S_i}\rvert\p{S_i\setminus\cp{p}}+\theta_{S_i}(p).
\end{align*}
\item[3)] There are isometries $f_i\colon S_i\to S_i$ and points $p_{j,n}\in X_n$ such that $\conv{p_{j,n}}{f_i\p{p_j}}$ and $\theta_{X_n}\p{p_{j,n}}=\theta_{S_i}\p{p_j}$ for every $p_j\in S_i$. 
\item[4)] Every loop in $X_n$ of diameter at most $\frac{\varepsilon}{2}$ is contractible in $X_n$. 
\end{itemize}
\end{prop}
\begin{proof}
Proceeding as in the proof of 
Theorem \ref{trm_length_surf_curv}, we find convex closed topological discs $D_{i,n}\subset S_i\setminus F$ satisfying the following properties: The point $p$ lies in the interior of the disc, $J_{i,n}\coloneqq \partial D_{i,n}$ is free of peak points and the absolute rotation of $J_{i,n}$, measured from the side of $S_i\setminus D_{i,n}$, \comp{pp. 272, 308}{AZ67} is bounded by some constant that does not depend on $n$. Moreover we have $\lim\limits_{n\to\infty}l_{i,n}\coloneqq length\p{J_{i,n}}=0$, $r_n\coloneqq l_{2,n}-l_{1,n}>0$ and $\lim\limits_{n\to\infty}diam\p{D_{i,n}}=0$.\\
From the Gauss-Bonnet theorem we derive $\tau\p{J_{i,n}}\le\theta_{S_i}(p)+\lvert\omega_{S_i}\rvert\p{D_{i,n}^0\setminus\cp{p}}$ where the rotation is measured from the side of $p$.\\
Let now $T_n\subset\mathbb{R}^2$ be an isosceles trapezoid such that its basis has the lengths $l_{1,n}$ and $l_{2,n}$ and its hight is equal to $\sqrt{r_n}$. Moreover let $Z_n$ be the topological cylinder obtained by gluing the legs of $T_n$ along an isometry. We denote the boundary components of this cylinder by $b_{1,n}$ and $b_{2,n}$ where $length\p{b_{i,n}}=l_{i,n}$.\\ Furthermore we equip $X_{i,n}\coloneqq S_i\setminus D_{i,n}^0$ with its induced length metric and glue the disjoint union $X_{1,n}\sqcup Z_n\sqcup X_{2,n}$ along length preserving homeomorphisms $f_i\colon J_{i,n}\to b_{i,n}$. We denote the constructed space by $X_n$.\\
It follows $\GH{X_n}{X}$ and that $X_n$ is homeomorphic to $S_1$.\\ From the gluing theorem in \cite[p. 289]{AZ67} we derive that $X_n$ is a surface with bounded curvature. In addition the theorem implies that, after passing to a subsequence, the desired inequality for $\lim\limits_{n\to\infty}\lvert\omega_{X_n}\rvert\p{X_n}$ holds.\\ Moreover we may assume the points $p_{j,n}\coloneqq p_j\in X_n$ to be convergent and we find isometries $f_i\colon S_i\to S_i$ such that $\conv{p_{j,n}}{f_i\p{p_j}}$ for every $p_j\in S_i$. In particular, we have  $\theta_{X_n}\p{p_{j,n}}=\theta_{S_i}\p{p_j}$.\\
From an argument in \cite[p. 14]{Dott24b} we see that every simple closed curve in $X_n$ is arbitrarily Hausdorff close and homotopic to some simple closed curve that intersects $J_{1,n}$ only finitely many times. Since every loop in $S_1$ of diameter at most $\varepsilon$ is contractible in $S_1$, after passing to a subsequence, the same applies to $X_n$ with the constant $\frac{\varepsilon}{2}$. 
\end{proof}
\end{lem}
\noindent
The construction in our next proof uses flat cylinders in $\mathbb{R}^3$. Therefore we introduce the following notation:
\begin{align*}
&\bar{Z}_{n,l}\coloneqq\cp{\p{x,y,z}\in\mathbb{R}^3\colon x^2+y^2=2^{-n},\ 0\le \lvert z\rvert\le\frac{l}{2}},\\
&D_{n,l}\coloneqq\cp{\p{x,y,z}\in\mathbb{R}^3\colon x^2+y^2\le 2^{-n},\ \lvert z\rvert=\frac{l}{2}}.
\end{align*}
Moreover we equip $Z_{n,l}\coloneqq\bar{Z}_{n,l}\cup D_{n,l}$ with its induced length metric.\\ 
We state the final approximation step:
\begin{lem}
Let $X$ be a space as in the approximating sequence from Lemma \ref{lem_approx_finite} such that every loop in $X$ of diameter at most $\varepsilon$ is contractible in $X$. Then $X$ can be obtained as the Gromov-Hausdorff limit of spaces $X_n\in\mathcal{M}\p{S,\frac{\varepsilon}{2},C_n}$ where $\lim\limits_{n\to\infty} C_n\le C$.
\begin{proof}
We denote the wedged intervals by $I_1,\ldots,I_k$ and define $l_q\coloneqq length(I_q)$. Next we cut out every $I_q$ and glue in the flat cylinder $Z_{n,l_q}$ instead. The gluing proceeds along the points $\p{0,0,\pm\frac{l_q}{2}}$. We denote the constructed space by $Y_n$.\\
Now we successively "repair" the wedge points $p\in Y_n$ by applying the construction from the proof of Proposition \ref{prop_approx_wedge_sum}. However, we have one more rule to follow: If $p$ lies in the "lid" of some glued-in cylinder, we choose the curve $J_{i,n}$ from the aforementioned construction as the boundary of the "lid".\\
After each wedge point has been eliminated, we end up with a new space $X_n$. In particular, we have $\GH{X_n}{X}$. Finally the proof of Proposition \ref{prop_approx_wedge_sum} yields that, after passing to a subsequence, we have $X_n\in\mathcal{M}\p{S,\frac{\varepsilon}{2},C_n}$ where $\lim\limits_{n\to\infty} C_n\le C$.
\end{proof}
\end{lem}
\noindent 
We note that $X_n$ can be obtained as the uniform limit of spaces in $\mathcal{R}\p{S,C_{n,k}}$ where $\lim\limits_{k\to\infty} C_{n,k}\le C_n$ \comp{p. 144}{Res93}.
Hence Theorem \ref{trm_converse} is a direct consequence of the lemmas in this subsection.
\bibliography{literature}
\bibliographystyle{abbrv}
\footnotesize{\textsc{Institute of Algebra and Geometry, Karlsruhe Institute of Technology, 76131 Karlsruhe, Germany}\\
\emph{E-mail}: \url{tobias.dott@kit.edu}
\end{document}